\documentclass[A4paper]{article}
\usepackage{amsfonts}
\usepackage[english]{babel}
\usepackage{latexsym}
\usepackage[latin1]{inputenc}
\usepackage{amssymb}
\usepackage{amsmath}
\usepackage{bbm}
\usepackage{amsthm}
\usepackage{graphicx}
\usepackage[usenames]{color}

\newtheorem{thm}{Theorem}
\newtheorem{lem}[thm]{Lemma}
\newtheorem{defi}[thm]{Definition}

\newcommand{\pd}[2]{\frac{\partial #1}{\partial #2}}
\newtheorem{nota}[thm]{Remark}

\title{On the finite element method for a nonlocal degenerate parabolic problem }
\author{Rui M.P. Almeida \thanks{Department of Mathematics, Faculty of Science,
University of Beira Interior, http://www.mat.ubi.pt/\~{}ralmeida, email: ralmeida@ubi.pt}
\and Stanislav N. Antontsev \thanks{Center for Mathematics and Fundamental Applications, Faculty of Science,
University of Lisbon, http://cmaf.ptmat.fc.ul.pt, email: anton@ptmat.fc.ul.pt}
\and Jos\'{e} C.M. Duque \thanks{Department of Mathematics, Faculty of Science,
University of Beira Interior, http://www.mat.ubi.pt/\~{}jduque, email: jduque@ubi.pt}}

\date{\today}

\begin{document}
\maketitle

\begin{abstract}
The aim of this paper is the numerical study of a class  of nonlinear nonlocal degenerate parabolic equations. The convergence and error bounds of the solutions are proved for a linearized Crank-Nicolson-Galerkin finite element method with polynomial approximations of degree $k\geq 1$. Some explicit solutions are obtained and used to test the implementation of the method in Matlab environment.

\textbf{keywords}: nonlocal; degenerate; parabolic; PDE; finite element.
\end{abstract}

\section{Introduction}

In this work, we study parabolic problems with nonlocal nonlinearity of the
following type
\begin{equation}  \label{prob}
\left\{
\begin{array}{l}
\displaystyle u_t-\left( \int_{\Omega}u^2(x,t)dx\right)^{\gamma} \Delta u=f\left(x,t\right)\,, \quad (x,t)\in \Omega\times ]0,T],\\
\displaystyle u\left(x,t\right) =0\,,\quad (x,t)\in \partial\Omega\times ]0,T], \\
\displaystyle u(x,0)=u_{0}(x)\,, \quad x\in \Omega, \\
\end{array}
\right. \,
\end{equation}
where $\Omega$ is a bounded open domain in $\mathbbm{R}^d$, $d=1,2,3$, $\gamma$ is a real constant, and $f$ and $u_0$ are continuous integrable functions.\\
In 1996, Chipot and Lovat \cite{CL97} proposed the equation
\begin{equation}\label{eq_cl}
u_t-a(\int_{\Omega} u \ dx)\Delta u =f
\end{equation}
to model the density of a population subject to spreading or heat propagation, and proved the existence and uniqueness of weak solutions. Since then, the existence, uniqueness and asymptotic behaviour of weak and strong solutions of parabolic equations and systems with nonlocal diffusion terms, have been widely studied (see, for example, \cite{ZC05,DAAFppa,RACFppa} and their references).\\
Concerning the numerical treatment of this equations we refer to some relevant works. Ackleh and Ke \cite{AK00} studied the problem
$$\left\{\begin{array}{l}
u_t=\frac{1}{a(\int_{\Omega}u\ dx)}\Delta u+f(u)\,, \quad (x,t)\in\Omega\times ]0,T],\\
u(x,t)=0\,, \quad (x,t)\in\partial \Omega\times ]0,T],\\
u(x,0)=u_0(x)\,, \quad x\in\overline \Omega,
\end{array}\right.$$
with $a(\xi)>0$ for all $\xi\neq0$, $a(0)\geq 0$ and $f$ Lipschitz-continuous satisfying $f(0)=0$.
In addition to the proof of the existence and uniqueness of solutions to this problem and the establishment of conditions on $u_0$ for the extinction in finite time and for the persistence of solutions, they also made some numerical simulations with a finite difference scheme in one dimension and a finite volume discretization in two space dimensions. In 2009, Bendahmane and Sepulveda \cite{BS09} used the model
\begin{equation}\label{eq_bs}
	\begin{cases}
		&\displaystyle  (u_1)_t-a_1\left(\int_\Omega u_1\ dx\right) \Delta u_1= -\sigma(u_1,u_2,u_3)-\mu u_1,\\
		&\displaystyle (u_2)_t-a_2\left(\int_\Omega u_2\ dx\right) \Delta u_2 = \sigma(u_1,u_2,u_3)-\gamma u_2-\mu u_2,\\
		&\displaystyle (u_3)_t-a_3\left(\int_\Omega u_3\ dx\right)\Delta u_3= \gamma u_2,
	\end{cases}	
\end{equation}
to investigate the propagation of an epidemic disease, in a physical domain $\Omega \subset \mathbbm{R}^{d}$ $(d=1,2,3)$. They established the existence of discrete solutions to a finite volume scheme and its
convergence to the weak solution of the PDE.
In \cite{DAAFppb}, the authors proved the optimal order of convergence for a linearized Euler-Galerkin
finite element method for a nonlocal system with absorbtion,
\begin{equation}\label{eq_daaf}
\left\{\begin{array}{l}
u_t-a_1(l_1(u),l_2(v))\Delta u+\lambda_1|u|^{p-2}u=f_1(x,t)\,, \quad (x,t)\in\Omega\times ]0,T],\\
v_t-a_2(l_1(u),l_2(v))\Delta v+\lambda_2|v|^{p-2}v=f_2(x,t)\,, \quad (x,t)\in \Omega\times ]0,T],\\
\end{array}\right.
\end{equation}
and presented some numerical results.
In \cite{RACFppb}, Robalo et al. obtained approximate numerical solutions for a nonlocal reaction-diffusion system, in a domain with moving boundaries, of the type
\begin{equation}\label{eq_racf}
\left\{\begin{array}{l}
u_t-a_1\left(\int_{\Omega(t)} v\ dx\right) u_{xx}=f_1(x,t)\,, \quad (x,t)\in\hat Q,\\
v_t-a_2\left(\int_{\Omega(t)} u\ dx\right) v_{xx}=f_2(x,t)\,, \quad (x,t)\in\hat Q,\\
\end{array}\right.
\end{equation}
where $\hat Q=\{(x,t)\in\mathbbm{R}^2:\alpha(t)<x<\beta(t), 0<t<T\}$, with a Matlab code based on the moving finite element method (MFEM) with high degree local approximations.
Almeida et al. \cite{ADFRpp,ADFRppb}, established the convergence and error bounds of the fully discrete solutions for a class of nonlinear equations and for systems of reaction-diffusion nonlocal type with moving boundaries, using a linearized Crank-Nicolson-Galerkin finite element method with polynomial approximations of any degree.\\
In this paper, we analyse a different diffusion term, dependent on the $L_2$-norm of the solution. In most of the previous papers, it is assumed that the diffusion term is bounded, with $0<m\leq a(s)\leq M<\infty$, $s\in\mathbbm{R}$, and so the problem is always nondegenerate. Here, we study a case were the diffusion term could be zero or infinity.
Problem (\ref{prob}) was studied in \cite{AAD14ppb}, where the authors proved the existence of weak solutions for $t\in[0,T]$ and the existence of a positive instant $t^*$ such that these solutions are unique and classical for $t\in[0,t^*]$.  In \cite{AAD14ppb}, the asymptotic behaviour of the solutions as time increases, was also studied.\\
 This work is concerned with the study of the convergence of a total discrete solution using a Crank-Nicolson-Galerkin finite element method and the use of this method to analyse the behaviour of the weak solutions. To the best of our knowledge, these results are new for nonlocal reaction-diffusion equations with this type of diffusion term.
The remaining of this paper is organized as follows. In Section 2, we formulate the problem and recall some useful definitions and lemmas. In Section 3, we define and prove the convergence of the semidiscrete
solution. Section 4 is devoted to the definition and proof of the convergence of a fully
discrete solution. In Section 5, we obtain some explicit solutions and analyse their behaviour, and then we use the deduced explicit solutions to simulate some examples in Section 6. Finally, in Section 7, we draw some conclusions.

\section{Statement of the problem}
Let $\Omega$ be a bounded open domain in $\mathbbm{R}^d$, $d=1,2,3$, with Lipschitz-continuous boundary $\partial\Omega$, and $T$ an arbitrary positive finite instant. We consider the problem of finding the function $u(x,t)$ which satisfies the following conditions
\begin{equation}  \label{probi}
\left\{
\begin{array}{l}
\displaystyle u_t- a(u)\Delta u=f\left(x,t\right)\,, \quad (x,t)\in \Omega\times ]0,T],\\
\displaystyle u\left(x,t\right) =0\,,\quad (x,t)\in \partial\Omega\times ]0,T], \\
\displaystyle u(x,0)=u_{0}(x)\,, \quad x\in \Omega, \\
\end{array}
\right. \,
\end{equation}
where $a(u)= \left(\int_{\Omega}u^2(x,t)dx\right)^{\gamma}$ with $\gamma\in \mathbbm{R}$ and $f$ and $u_0$ are continuous integrable functions.\\
In what follows, $(\cdot, \cdot)$ and $\|\cdot\|$ denote, respectively, the inner product and the
norm in $L_2(\Omega)$, and $C$ represents a constant, but not always the same value. \\
The definition of a weak solution to this problem is as follows:
\begin{defi}[Weak solution]\label{fraca}
We say that the function $u$ is
a weak solution of Problem (\ref{probi}) if
\begin{equation}
u\in L_{2}(0,T;H_{0}^{1}(\Omega)\cap H^2(\Omega)),\frac{\partial u}{\partial t}\in L_{2}(0,T;L_{2}(\Omega)),  \label{regx}
\end{equation}
the equality
\begin{equation}\label{fracav}
(u_t,w)+a(u)(\nabla u,\nabla w)=(f,w)
\end{equation}
is valid for all $w\in H_{0}^{1}(\Omega)$ and $t\in ]0,T[$, and
\begin{equation}\label{condi}
u(x,0)=u_{0}(x),\quad x\in \Omega.
\end{equation}
\end{defi}
The existence and uniqueness of a weak solution, in the sense of this definition, were proved in \cite{AAD14ppb}, and follow mainly from the lemmas below. These lemmas prove the nondegeneracy and the Lipschitz-continuity of the diffusion term and will be needed in the proofs of the theorems in the following sections.
\begin{lem}\label{lema1_t*}
Suppose that $\gamma>0$. If $u_0\in H_0^1(\Omega)$, $f\in L_2(0,T;H_0^1(\Omega))$ and $\int_{\Omega} u_0\ dx>0$, then there exists a $t^*>0$ such that $a(u)\geq m>0$ for $t\in[0,t^*]$, where $u$ is a weak solution of Problem (\ref{probi}).
\end{lem}
\begin{lem}\label{lema2_t*}
If $u_0\in H_0^1(\Omega)$, $f\in L_2(0,T;L_2(\Omega))$, $\int_{\Omega} u_0\ dx>0$ and $\gamma< 0$, then there exists a $t^*>0$ such that $a(u)\leq M<\infty$ for $t\in[0,t^*]$, with $u$ a weak solution of Problem (\ref{probi}).
\end{lem}
\begin{lem}\label{Lip_a}
If
$$0<m \leq \int_{\Omega} v^2\ dx, \int_{\Omega} w^2\ dx\leq M<\infty,$$
then
$$|a(v)-a(w)|\leq C \|v-w\|,$$
where $C$ may depend on $\gamma$ , $m$ and $M$.
\end{lem}
The proof of this lemmas can be found in \cite{AAD14ppb}.

\section{Space discretization}
In this section, we discretize the spatial domain into simplexes (intervals in one dimension, triângles in two dimensions and triangular pyramids in three dimensions). We also define some auxiliary functions as well as the semidiscrete solution. Then we prove the convergence of the semidiscrete solution to the weak solution.
Let $\mathcal{T}_h$ denote a partition of $\Omega$ into disjoint simplexes $
T_i$, $i=1,\dots,n_t$, such that no vertex of any simplex lies in the interior or on the side of another simplex and $h=\max\{diam(T_i),i=1,\dots,n_t\}$. Let $S_h^k$ denote the continuous functions on the
closure $\bar\Omega$ of $\Omega$, which are polynomials of degree $k$ in each
simplex of $\mathcal{T}_h$ and which vanish on $\partial\Omega$, that is,
\begin{equation*}
S_h^k=\{W\in C_0^0(\bar\Omega)| {W}_{|{T_i}} \text{ is a polynomial of
degree $k$ for all } T_i\in\mathcal{T}_h\}.
\end{equation*}
If $\{\varphi_j\}_{j=1}^{n_p}$ is the Lagrange basis of $S_h^k$, associated to the equally spaced nodes $\{P_j\}_{j=1}^{n_p}$, then we can represent every $W\in S_h^k$ as $W=\sum_{j=1}^{n_p}w_j\varphi_j$. The definitions and lemmas below are important in the proofs of the main theorems. The proofs of the lemmas can be found in \cite{Tho06}. Given a smooth function $u$ on $\Omega$ which vanishes on $\partial\Omega$,
we may define its interpolant and its projection as follows.
\begin{defi}[Interpolant]
A function $I_hu\in S_h^k$ is said to be the interpolant of $u\in
H_0^1(\Omega)$ in $S_h^k$ if it satisfies
\begin{equation*}
I_hu=\sum_{j=1}^{n_p}u(P_j)\varphi_j.
\end{equation*}
\end{defi}
\begin{lem}
\label{errint} If $u\in H^{k+1}(\Omega)\cap H_0^1(\Omega)$, then
\begin{equation*}
\|I_h u-u\|+h\|\nabla(I_h u-u)\|\leq Ch^{k+1}\|u\|_{H^{k+1}(\Omega)}.
\end{equation*}
\end{lem}
\begin{defi}[Ritz projection]
A function $\tilde U(x,t)\in S_h^k$ is said to be the Ritz projection of $u\in
H_0^1(\Omega)$ onto $S_h^k$ if it satisfies
\begin{equation*}
(\nabla \tilde U,\nabla W)=(\nabla u,\nabla W),\quad \text{ for all } W\in S_h^k.
\end{equation*}
\end{defi}
\begin{lem}
\label{errproj} If $u\in H^{k+1}(\Omega)\cap H_0^1(\Omega)$, then
\begin{equation*}
\|\tilde U-u\|+h\|\nabla(\tilde U-u)\|\leq Ch^{k+1}\|u\|_{H^{k+1}(\Omega)},
\end{equation*}
where $C$ does not depend either on $h$ or $k$.
\end{lem}
The semidiscrete problem consists in
finding $U$, belonging to $S_h^k$, for $t\geq 0$, that satisfies Definition \ref{fraca}.
\begin{defi}[Semidiscrete solution]
A function $U(x,t)\in S_h^k$, for $t\geq 0$, is said to be a semidiscrete solution of Problem (\ref{probi}) if it satisfies
\begin{equation}  \label{probsd}
\left\{%
\begin{array}{l}
({U}_{t},W)+a(U)(\nabla U,\nabla W)=(f,W),\\
U(x,0)=I_hu_0, \\
\end{array}%
\right.
\end{equation}
for all $W\in
S_h^k$ and $t\in]0,T[.$
\end{defi}
Since $a(U)$ is continuous, the existence of a semidiscrete solution can easily be proved using Caratheodory's theorem. The proof of the uniqueness of the semidiscrete solution is identic to the proof of Theorem 14 in \cite{AAD14ppb} and the stability is proved using the arguments of Lemmas 2 and 4 in  \cite{AAD14ppb}.
The convergence of the semidiscrete solution to the weak solution, as $h$ tends to zero, is proved in the next theorem.
\begin{thm}\label{conv_h}
Suppose that $u_0\in H_0^1(\Omega)$, $f\in L_2(0,t^*;H_0^1(\Omega))$ and $\int_{\Omega} u_0\ dx>0$. If $u$ is the weak solution of Problem (\ref{probi}) and $U$ is its semidiscrete solution, then
\begin{equation*}
\|U-u\| \leq Ch^{k+1},\quad t\in]0,t^*],
\end{equation*}
where $C$ does not depend on $h$ and $k$, but may depend on $\|\nabla u\|_{L_{\infty}(0,t^*;L_2(\Omega))}$, $\|u_0\|_{H^{k+1}(\Omega)}$, $\|u\|_{L_2(0,t^*;H^{k+1}(\Omega))}$ and $\|u_t\|_{L_2(0,t^*;H^{k+1}(\Omega))}$.
\end{thm}
In virtue of Lemmas \ref{lema1_t*}, \ref{lema2_t*} and \ref{Lip_a}, the proof follows from classical arguments (see, for example, \cite{Tho06}), and we will only present the main steps.
\begin{proof}
First, we split the error in two parts, by introducing the Ritz projection $\tilde U $ of $u$, and we obtain
$$\|U-u\|\leq\|U-\tilde U\|+\|\tilde U-u\|=\|\theta\|+\|\rho\|.$$
The estimate of $\rho$ is obtained by Lemma \ref{errproj}, as
\begin{equation}\label{erho}
\|\rho\|\leq Ch^{k+1}\|u\|_{H^{k+1}(\Omega)}.
\end{equation}
For $\theta$, we have\\
$({\theta}_{t},W)+a(U)(\nabla\theta,\nabla W)$\\

$=(U_{t},W)+a(U)(\nabla U,\nabla W)-(\tilde U_{t},W)-a(U)(\nabla\tilde U,\nabla W)$\\

$=(f,W)-((u)_{t},W)-a(u)(\nabla u,\nabla W)+((u-\tilde U)_{t},W)$\\

$\phantom{=}+(a(u)-a(U))(\nabla u,\nabla W)$\\

$=((u-\tilde U)_{t},W)+(a(u)-a(U))(\nabla u,\nabla W).$\\
Using $\theta$ as the test function $W$, we arrive at
$$\frac12\frac{d}{dt}\|\theta^2\|^2+a(U)\|\nabla{\theta}\|^2=
({\rho}_{t},\theta)+(a(u)-a(U))(\nabla u,\nabla\theta).$$
Thus, the H\"older inequality and Lemmas \ref{lema1_t*} and \ref{lema2_t*} imply that
$$\frac12\frac{d}{dt}\|\theta\|^2+m\|\nabla\theta\|^2\leq\frac12\|{\rho}_{t}\|^2
+\frac12\|\theta\|^2+m\|\nabla\theta\|^2+C(a(u)-a(U))^2
\|\nabla u\|^2.$$
The Lipschitz-continuity of the diffusion term stated in Lemma \ref{Lip_a} permits us to prove that $\theta$ satisfies the differential inequality
$$\frac{d}{dt}\|\theta\|^2\leq C\|\theta\|^2+C\|{\rho}\|^2+\|{\rho}_{t}\|^2,$$
with $C=C(\|\nabla u\|_{L_{\infty}(0,t^*;L_2(\Omega))})$.
By Gronwall's Lemma, we obtain
$$\|\theta\|^2\leq C\|\theta(x,0)\|^2+C\int_0^{t^*}\|{\rho}\|^2\ dt+\int_0^{t^*}\|{\rho}_{t}\|^2\ dt.$$
Making use of Lemmas \ref{errint} and \ref{errproj}, the elements of the right hand side are bounded as follows:
$$\|\theta(x,0)\|^2\leq Ch^{2(k+1)}\|u_0\|_{H^{k+1}(\Omega)}^2,$$
$$\int_0^{t^*}\|{\rho}\|^2\ dt\leq Ch^{2(k+1)}\int_0^{t^*}\left\|u\right\|_{H^{k+1}(\Omega)}^2\ dt,$$
$$\int_0^{t^*}\|{\rho}_{t}\|^2\ dt\leq Ch^{2(k+1)}\int_0^{t^*}\left\|u_t\right\|_{H^{k+1}(\Omega)}^2\ dt.$$
Then $\theta$ satisfies
$$\|\theta\|^2\leq Ch^{2(k+1)},$$
where $C$ depends on $\|\nabla u\|_{L_{\infty}(0,t^*;L_2(\Omega))}$, $\|u_0\|_{H^{k+1}(\Omega)}$, $\|u\|_{L_2(0,t^*;H^{k+1}(\Omega))}$ and $\|u_t\|_{L_2(0,t^*;H^{k+1}(\Omega))}$. Adding the estimate in (\ref{erho}), the result is proved.
\end{proof}

\section{Time discretization}
Now we discretize the problem also in time. For the time discretization, we will use the Crank-Nicolson method. In order to avoid the need to solve a nonlinear system in each time step, we will linearise the method by transforming it in a multistep method. For the first estimate we will use a predictor-corrector scheme.
Consider the partition in non empty intervals $[0,t^*]=\overset{n_{i}}{
\underset{j=1}{\cup }}[t_{j-1},t_{j}]=\overset{n_{i}}{\underset{j=1}{\cup }
}I_{j}$, with $int(I_{j})\cap
int(I_{i})=\emptyset$, $\forall i\neq j$ and $\delta=\max_{j=1,...,n_i}\{t_j-t_{j-1}\}$.  Define
$$\bar\partial U_n=\frac{U_n-U_{n-1}}{\delta},\, \hat U_n=\frac{U_n+U_{n-1}}{2},\, \bar U_n=\frac32U_{n-1}-\frac12U_{n-2}$$
$$\text{ and } f_{n-1/2}=f(x,\frac{t_n+t_{n-1}}{2}).$$
The fully discrete approximation $U_n(x)\approx u(x,t_n)$, $n=0,\dots,n_{i}$, belonging to $S_h^k$, is defined as follows:
\begin{defi}[Fully discrete approximation]
A function $U_n(x)\in S_h^k$ is said to be a fully discrete solution of Problem (\ref{probi}) if it satisfies\\

$U_0=I_hu_0, \quad n=0$,
\begin{equation}\label{sol_dise}
(\frac{U_{1,0}-U_0}{\delta},W)+a(U_0)(\nabla \left( \frac{U_{1,0}+U_0}{2}\right),\nabla W)=(f_{1/2},W), \quad n=1,
\end{equation}
\begin{equation}\label{sol_disc}
(\bar \partial U_1,W)+a(\frac{U_{1,0}+U_0}{2})(\nabla \hat U_1,\nabla W)=(f_{1/2},W), \quad n=1,
\end{equation}
\begin{equation}\label{sol_dis}
(\bar \partial U_n,W)+a(\bar U_n)(\nabla \hat U_n,
\nabla W)=(f_{n-1/2},W), \quad n=2,\dots,n_{i},
\end{equation}
for all $W\in S_h^k$.
\end{defi}
We observe that the linear systems in (\ref{sol_dise})-(\ref{sol_dis}) always have a unique solution
In the next theorem, we prove the convergence of the fully discrete solution to the weak solution.
\begin{thm}\label{conv_d}
Suppose that $u_0\in H_0^1(\Omega)$, $f\in L_2(0,t^*;H_0^1(\Omega))$ and $\int_{\Omega} u_0\ dx>0$. If $u$ is the solution of Problem (\ref{probi}) and $U_n$ the fully discrete solution, then
\begin{equation*}
\|U_n(x)-u(x,t_n)\| \leq C(h^{k+1}+\delta^2),\quad n=1,\dots,n_{i},
\end{equation*}
where $C$ does not depend on either $h$ or $k$ nor on $\delta$, but may depend on $\left\|\pd{^3u}{t^3}\right\|$, $\|u\|_{H^{k+1}(\Omega)}$, $\|u_t\|$, $\|u_{tt}\|$ and $\|\nabla u_{tt}\|$.
\end{thm}

\begin{proof}
As before, we split the error as follows:
$$\|U_n(x)-u(x,t_n)\|\leq\|U_n(x)-\tilde U_n\|+\|\tilde U_n(x)-u(x,t_n)\|=\|\theta_n\|+\|\rho_n\|.$$
The estimate of $\rho_n$ is obtained by Lemma \ref{errproj}.
Concerning $\theta_n$, we start by the estimator's solution. Considering $\theta_{1,0}=U_{1,0}-\tilde U_1$, $\hat
\theta_{1,0}=\frac{\theta_{1,0}+\theta_0}{2}$ and
$\overline{\partial}\theta_{1,0}=\frac{\theta_{1,0}+\theta_0}{\delta}$.
We then have
$$(\bar\partial \theta_{1,0},W)+a(U_0)(\nabla \hat \theta_{1,0},\nabla W)=(\bar\partial U_{1,0},W)+a(U_0)(\nabla \hat U_{1,0},\nabla W)-(\bar\partial \tilde U_{1,0},W)$$
$$-a(U_0)(\nabla \hat{\tilde{U}}_1,\nabla W)=(f_{1/2},W)-((u_t)_{1/2},W)-a(u_{1/2})(\nabla u_{1/2},\nabla W)$$
$$+((u_t)_{1/2}-\bar\partial \tilde{U}_1,W)+(a(u_{1/2})\nabla u_{1/2}-a(U_0)\nabla \hat u_1,\nabla W) $$
Choosing $W=\hat{\theta}_{1,0}$, we obtain
$$\frac12\bar \partial\|\theta_{1,0}\|^2+m\|\nabla\hat \theta_{1,0}\|^2\leq C(\|(u_t)_{1/2}-\bar\partial \tilde{U}_1\|+\|\nabla(u_{1/2}- \hat u_1)\|+\|u_{1/2}-U_0\|)\|\nabla\hat \theta_{1,0}\|.$$
Using the differentiation and interpolation errors, we can estimate each element of the right hand side as\\

$\|(u_t)_{1/2}-\bar\partial\tilde U_1\|\leq\|(u_t)_{1/2}-\overline\partial u_1\|+\|\overline\partial u_1-\overline\partial\tilde U_1\| \leq C\delta^2\left\|\pd{^3u}{t^3}\right\|+C h^{k+1}\|u\|_{H^{k+1}(\Omega)},$\\

$\left\|\nabla(u_{1/2}-\hat{u}_1)\right\|\leq C\delta^2\|\nabla u_{tt}\|,$\\

$\|u_{1/2}-U_0\|\leq\|u_{1/2}-u_0\|+\|u_0-U_0\|\leq C\delta\|u_t\|+C h^{k+1}\|u\|_{H^{k+1}(\Omega)}.$\\

\noindent Hence
$$\overline{\partial}\|\theta_{1,0}\|^2\leq C (h^{k+1}+\delta)^2,$$
and we have the estimate
$$\|\theta_{1,0}\|^2\leq\|\theta_{0}\|^2+ C\delta (h^{k+1}+\delta)^2\leq C(h^{2(k+1)}+\delta^3),$$
with $C=C(\left\|\pd{^3u}{t^3}\right\|, \|u\|_{H^{k+1}(\Omega)}, \|u_t\|, \|\nabla u_{tt}\|)$.
Repeating this process for the corrector equation, we arrive at\\

$\frac12\bar \partial\|\theta_{1}\|^2+m\|\nabla\hat \theta_{1}\|^2\leq C(\|(u_t)_{1/2}-\bar\partial \tilde{U}_1\|+\|\nabla(u_{1/2}- \hat u_1)\|$\\

$\phantom{\frac12\bar \partial\|\theta_{1}\|^2+m\|\nabla\hat \theta_{1}\|^2\leq}+\|u_{1/2}-\frac{U_{1,0}-U_0}{2}\|)\|\nabla\hat \theta_{1}\|,$\\
and now we use the estimate
$$\|u_{1/2}-\frac{U_{1,0}-U_0}{2}\|\leq \|u_{1/2}-\hat{\tilde U}_1\|+\|\hat{\tilde U}_1-\frac{U_{1,0}-U_0}{2}\|$$
$$\leq \|u_{1/2}-\hat{\tilde U}_1\|+\frac12\|\theta_{1,0}\|+\frac12\|\theta_0\|$$
$$\leq Ch^{k+1}\|u\|_{H^{k+1}(\Omega)}+C\delta^2\|u_{tt}\|+C(h^{k+1}+\delta^{\frac32})+Ch^{k+1}\|u\|_{H^{k+1}(\Omega)}$$
$$\leq C(h^{k+1}+\delta^{\frac32}),$$
and, by Cauchy's inequality, we conclude that
$$\overline{\partial}\|\theta_{1}\|^2\leq C (h^{2(k+1)}+\delta^3).$$
Whence
$$\|\theta_{1}\|^2\leq\|\theta_{0}\|^2+ C\delta (h^{2(k+1)}+\delta^3)\leq C(h^{2(k+1)}+\delta^4),$$
where  $C=C(\left\|\pd{^3u}{t^3}\right\|, \|u\|_{H^{k+1}(\Omega)}, \|u_{tt}\|, \|\nabla u_{tt}\|)$.
Using the estimate
\begin{eqnarray*}
   \|u_{n-1/2}-\overline{U}_n\|& \leq &  \|u_{n-1/2}-\overline{u}_n\|+\|\overline{u}_n-\overline{U}_n\|\\
   & \leq & \|u_{n-1/2}-\overline{u}_n\|+\|\overline\rho_n\|+\|\overline\theta_n\| \\
   & \leq & C\delta^2\|u_{tt}\|+C h^{k+1}\|u\|_{H^{k+1}(\Omega)}+C(\|\theta_{n-1}\|+\|\theta_{n-2}\|),
\end{eqnarray*}
and applying the same process to Equation (\ref{sol_dis}), we can show that
$$\frac12\bar \partial\|\theta_{n}\|^2+m\|\nabla\hat \theta_{b}\|^2\leq C(\|(u_t)_{n-1/2}-\bar\partial \tilde{U}_n\|+\|\nabla(u_{n-1/2}- \hat u_n)\|$$
$$+\|u_{n-1/2}-\bar U_n\|)\|\nabla\hat \theta_{n}\|,$$
and
$$\overline{\partial}\|\theta_{n}\|^2\leq C\|\theta_{n-1}\|^2+C\|\theta_{n-2}\|^2+C(h^{(k+1)}+\delta^2)^2,\quad n\geq 2.$$
Iterating, we obtain
$$\|\theta_{n}\|^2\leq(1+C\delta)\|\theta_{n-1}\|^2+ C\delta\|\theta_{n-2}\|^2 +C\delta (h^{k+1}+\delta^2)^2\leq C\|\theta_{1}\|^2+ C\delta\|\theta_{0}\|^2$$
$$+ C\delta (h^{k+1}+\delta^2)^2,$$
and recalling the estimates for $\|\theta_{0}\|$, $\|\theta_{1}\|$ and
$\|\rho_n\|$, the proof is complete.
\end{proof}

\section{Explicit solution}
In the present section we, will illustrate the theoretical results obtained with some numerical examples. In order to calculate the exact error we require the explicit exact solutions to Problem (\ref{probi}). For $\gamma=0$ there exist well known explicit, so we will consider the case $\gamma\neq 0$. We seek an explicit solution of the form
\begin{equation}\label{fvs}
u(x,t)=k(x)l(t).
\end{equation}
The first equation in (\ref{probi}) becomes
\begin{equation}\label{eqvs}
k(x)l'(t)-l^{2\gamma+1}(t)\left(\int_{\Omega}k^2(x)\ dx\right)^{\gamma} \Delta k(x)=f(x,t).
\end{equation}
If $l$ is chosen such that
\begin{equation}\label{eql}
l'(t)=-l^{2\gamma+1}(t)\Leftrightarrow l(t)=(2\gamma t-2\gamma C)^{-\frac{1}{2\gamma}},\, \gamma\neq 0,\, C\in\mathbbm{R},
\end{equation}
then (\ref{eqvs}) has the form
\begin{equation}\label{eqk}
k(x)+\left(\int_{\Omega}k^2(x)\ dx\right)^{\gamma} \Delta k(x)=\frac{f(x,t)}{-l^{2\gamma+1}(t)}.
\end{equation}
To obtain a function $k(x)$ which only depends on $x$, we must assume that
\begin{equation}\label{eqf}
\frac{f(x,t)}{-l^{2\gamma+1}(t)}=g(x)\Leftrightarrow f(x,t)=-g(x)l^{2\gamma+1}(t).
\end{equation}
In this case, let $w(x,\alpha)$ be such that
\begin{equation}\label{eqw}
w(x)+\alpha \Delta w(x)=g(x).
\end{equation}
Then
\begin{equation}\label{solk}
k(x)=w(x,\left(\int_{\Omega} w^2\ dx\right)^{\gamma})
\end{equation}
is a solution of (\ref{eqk}).
But (\ref{solk}) is defined in an implicit way. In order to obtain $k$ in an explicit form, we must solve the equation

\begin{equation}\label{eq_alpha}
\alpha=\left(\int_{\Omega}w^2(x,\alpha)\ dx\right)^{\gamma}.
\end{equation}
Collecting (\ref{solk}), (\ref{eql}) and (\ref{fvs}), we obtain an explicit solution for the first equation in (\ref{probi}).

\begin{nota}
The existence of conditions for the solvability of Equation (\ref{eq_alpha}) is under study.
\end{nota}

\subsection{One dimension}
For $d=1$, Equation (\ref{eqw}) becomes a linear second order ordinary differential equation
\begin{equation}\label{eqw1d}
w(x)+\alpha w''(x)=g(x).
\end{equation}
Since $\alpha>0$, the solution of the homogeneous equation $w(x)+\alpha w''(x)=0$ is
$$w_1(x)=C_1\sin(\frac{x}{\sqrt{\alpha}})+C_2\cos(\frac{x}{\sqrt{\alpha}}).$$
Using the variation-of-constants method, we will find a solution of the form
$$w(x)=v_1(x)\sin(\frac{x}{\sqrt{\alpha}})+v_2(x)\cos(\frac{x}{\sqrt{\alpha}})$$
satisfying
$$v_1'(x)\sin(\frac{x}{\sqrt{\alpha}})+v_2'(x)\cos(\frac{x}{\sqrt{\alpha}})=0,$$
and Equation (\ref{eqw1d}). So $v_1'$ and $v_2'$ are solutions of
$$\left[\begin{array}{cc}
\sin(\frac{x}{\sqrt{\alpha}})&\cos(\frac{x}{\sqrt{\alpha}})\\
\sqrt{\alpha}\cos(\frac{x}{\sqrt{\alpha}})&-\sqrt{\alpha}\sin(\frac{x}{\sqrt{\alpha}})
\end{array}\right]
\left[\begin{array}{c}
v_1'\\
v_2'
\end{array}\right]=
\left[\begin{array}{c}
0\\
g
\end{array}\right].$$
Solving this system, we obtain
$$v_1'(x)=\frac{1}{\sqrt{\alpha}}g(x)\cos(\frac{x}{\sqrt{\alpha}})\Rightarrow v_1(x)=C_1+\frac{1}{\sqrt{\alpha}}\int_0^xg(\xi)\cos(\frac{\xi}{\sqrt{\alpha}})\ d\xi$$
and
$$v_2'(x)=-\frac{1}{\sqrt{\alpha}}g(x)\sin(\frac{x}{\sqrt{\alpha}})\Rightarrow v_2(x)=C_2-\frac{1}{\sqrt{\alpha}}\int_0^xg(\xi)\sin(\frac{\xi}{\sqrt{\alpha}})\ d\xi.$$
Finally, if $g$ is continuous in $\Omega$, then Equation (\ref{eqw1d}) admits the solution
$$w(x)= \left(C_1+\frac{1}{\sqrt{\alpha}}\int_0^xg(\xi)\cos(\frac{\xi}{\sqrt{\alpha}})\ d\xi\right) \sin(\frac{x}{\sqrt{\alpha}})$$
\begin{equation}\label{sol_expl_w}
\phantom{w(x)= }+\left(C_2-\frac{1}{\sqrt{\alpha}}\int_0^xg(\xi)\sin(\frac{\xi}{\sqrt{\alpha}})\ d\xi\right) \cos(\frac{x}{\sqrt{\alpha}}).
\end{equation}
\begin{nota}
Constants $C$, $C_1$, $C_2$ and $\alpha$ must be chosen in such a way that $u$ satisfies the initial data, the boundary conditions and Equation (\ref{eq_alpha}).
\end{nota}

\subsection{Two dimensions}
If we assume that $f=0$, then, in 2D space domains, Equation (\ref{eqw}) becomes
$$w(x,y)+\alpha\left(\pd{^2w}{x^2}+\pd{^2w}{y^2}\right)=0.$$
Searching again for a solution with separate variables, that is,
$w(x,y)=X(x)Y(y),$
we obtain the equation
$$X(x)Y(y)+\alpha(X''(x)Y(y)+X(x)Y''(y))=0.$$
Then $X$ and $Y$ satisfy the condition
$$-\alpha\frac{X''(x)}{X(x)}=\frac{Y(y)+\alpha Y''(y)}{Y(y)}=\lambda=constant.$$
For $X$, we need to solve the second order linear equation
\begin{equation}\label{eqX}
-\lambda X(x)-\alpha X''(x)=0.
\end{equation}
If $\lambda>0$, then, since $\alpha>0$,  Equation (\ref{eqX}) has the solution
$$X(x)=A_1 \cos\left(\sqrt{\frac{\lambda}{\alpha}}x\right)+A_2 \sin\left(\sqrt{\frac{\lambda}{\alpha}}x\right).$$
The equation for $Y$ is
\begin{equation}\label{eqY}
(1-\lambda) Y(y)+\alpha Y''(y)=0.
\end{equation}
If $\lambda<1$, then, since $\alpha>0$, Equation (\ref{eqY}) has the solution
$$Y(y)=B_1 \cos\left(\sqrt{\frac{1-\lambda}{\alpha}}y\right)+B_2 \sin\left(\sqrt{\frac{1-\lambda}{\alpha}}y\right).$$
Then
$$w(x,y)=\left(A_1 \cos\left(\sqrt{\frac{\lambda}{\alpha}}x\right)+A_2 \sin\left(\sqrt{\frac{\lambda}{\alpha}}x\right)\right)$$
\begin{equation}\label{sol_w2d}
 \times\left(B_1 \cos\left(\sqrt{\frac{1-\lambda}{\alpha}}y\right)+B_2 \sin\left(\sqrt{\frac{1-\lambda}{\alpha}}y\right)\right)
\end{equation}
\begin{nota}
Constants $C$, $A_1$, $A_2$, $B_1$, $B_2$ and $\alpha$ must be chosen in such a way that $u$ satisfies the initial data, the boundary conditions and Equation (\ref{eq_alpha}).
\end{nota}

\subsection{Solution analysis}
Now it is interesting to analyse the type of solutions and their behaviour. First, we define the positive part of a function $f$ as
$$[f]_+=\left\{\begin{array}{ll}
f,&f>0\\
0,&f\leq0
\end{array}\right..$$
If $\gamma>0$ and
$$f(x,t)=-\frac{g(x)}{(2\gamma t-2\gamma C)^{\frac{1+2\gamma}{2\gamma}}},$$
then the solution is
$$u(x,t)=\frac{k(x)}{(2\gamma t-2\gamma C)^{\frac{1}{2\gamma}}}.$$
The constant $C$ is determined by $u_0$ and, in the case $g\neq 0$, by $f$. If $f$ is integrable in $\Omega$, then $k(x)$ is bounded. Considering $C=0$, we obtain a solution defined for $t>0$ and infinity at $t=0$, that is, we have a source type solution. If $C<0$, then $u$ is defined  in $t\geq 0$, but,
on the other hand, if $C>0$, both $f$ and $u$ are only defined for $t>C$. Making the change of variable $\tau=t+C+\varepsilon$, $\varepsilon>0$, we arrive at a solution similar to that in the case $C>0$. In each case, the solution tends to zero as time tends to infinity, as was proved in Theorem 18 in \cite{AAD14ppb}.\\
If $-\frac12<\gamma<0$ and
$$f(x,t)=-g(x)[2|\gamma| (C-t)]_+^{\frac{1+2\gamma}{2|\gamma|}},$$
then the corresponding solution is
$$u(x,t)=k(x)[2|\gamma| (C-t)]_+^{\frac{1}{2|\gamma|}}.$$
If $C\leq 0$, then $f$ and $u$ are zero for $t\geq0$. The case $C>0$ is more interesting because they are defined in $t\geq0$, but they are non zero only for $t\in[0,C[$, hence we have a finite time extinction phenomenon, like it was proved in Theorem 19 in \cite{AAD14ppb}.\\
In the case $\gamma=-\frac12$, which corresponds to $a(u)=\frac{1}{\|u\|}$, the solution is
$$u(x,t)=k(x)[C-t]_+,$$
when $f$ does not depend on $t$. Choosing $C\leq0$ the solution is zero for $t\geq0$, but choosing $C>0$, $u_0$ is nonzero and $u$ becomes extinct at $t=C$, even with a function $f$ that does not gets extinct. This example does not contradict Theorem 19 in \cite{AAD14ppb}.\\
For $\gamma<-\frac12$, the solution exhibits a curious behaviour. Indeed, if
$$f(x,t)=-\frac{g(x)}{[2|\gamma| (C-t)]_+^{\left|\frac{2\gamma+1}{2\gamma}\right|}},$$
then the solution is
$$u(x,t)=k(x)[2|\gamma| (C-t)]_+^{\frac{1}{2|\gamma|}}.$$
In the case $C>0$, we can observe an extinction of the solution at $t=C$, despite the fact that the function $f$ tends to infinity as $t$ tends to $C$.
\begin{nota}
If, in (\ref{eql}), we choose $l$ such that
$$l'(t)=l^{2\gamma+1}\Leftrightarrow l(t)=(-2\gamma t-2\gamma C)^{-\frac{1}{2\gamma}},$$
then the solution does not have the behaviour proved in \cite{AAD14ppb}. For example, in the one dimensional case, if $f=0$ and $\gamma>0$, then
$$u(x,t)=\frac{C_1 e^{\frac{x}{\sqrt{\alpha}}}+C_2 e^{-\frac{x}{\sqrt{\alpha}}}}{(-2\gamma t-2\gamma C)^{\frac{1}{2\gamma}}}$$
is a solution, but it blows up in finite time, which contradicts Theorem 18 in \cite{AAD14ppb}.
\end{nota}

\subsection{Example 1}
We consider Problem (\ref{probi}) in $]0,1[$ with $\gamma=\frac15$ and  $f(x,t)=\frac{x^2}{(t+1)^2}$, that is,
\begin{equation}\label{prob1}
\left\{
\begin{array}{l}
\displaystyle u_t- \left(\int_0^1u^2\ dx\right)^{\frac12}u_{xx}=\frac{x^2}{(t+1)^2}\,, \quad (x,t)\in ]0,1[\times ]0,10],\\
\displaystyle u(0,t)=u(1,t)=0\,,\quad t\in ]0,10], \\
\displaystyle u(x,0)=u_{0}(x)\,, \quad x\in ]0,1[. \\
\end{array}
\right. \,
\end{equation}
Since $\gamma=\frac12$, by Equation (\ref{eql}), we have
$$l(t)=(t-C)^{-1},$$
and, since $f(x,t)=\frac{x^2}{(t+1)^2}$, by (\ref{eqf}), we must consider $g(x)=-x^2$ and $C=-1$.
The equation that $w(x)$ must satisfy is
$$w(x)+\alpha w''(x)=-x^2,$$
and its solution is
$$w(x)= \left(C_1+\frac{1}{\sqrt{\alpha}}\int_0^x-\xi^2\cos(\frac{\xi}{\sqrt{\alpha}})\ d\xi\right) \sin(\frac{x}{\sqrt{\alpha}})+$$
$$\phantom{w(x)= }\left(C_2-\frac{1}{\sqrt{\alpha}}\int_0^x-\xi^2\sin(\frac{\xi}{\sqrt{\alpha}})\ d\xi\right) \cos(\frac{x}{\sqrt{\alpha}})$$
$$=C_1\sin(\frac{x}{\sqrt{\alpha}})+(C_2-2\alpha)\cos(\frac{x}{\sqrt{\alpha}})-x^2+2\alpha.$$
Thus
$$u(x,t)=\frac{C_1\sin(\frac{x}{\sqrt{\alpha}})+(C_2-2\alpha)\cos(\frac{x}{\sqrt{\alpha}})-x^2+2\alpha}{t+1}.$$
Imposing the boundary conditions, we obtain the values of $C_1$ and $C_2$ as
$$u(0,t)=0\Rightarrow C_2-2\alpha+2\alpha=0\Rightarrow C_2=0,$$
$$u(1,t)=0\Rightarrow C_1\sin(\frac{1}{\sqrt{\alpha}})-2\alpha \cos(\frac{1}{\sqrt{\alpha}})-1+2\alpha=0$$
$$\Rightarrow C_1=\frac{1-2\alpha+2\alpha\cos(\frac1{\sqrt{\alpha}})}{\sin(\frac1{\sqrt{\alpha}})}.$$
To finalise this procedure, we only need to solve Equation (\ref{eq_alpha}) which, in this case, is
$$\alpha=\left(\int_0^1\left(\frac{1-2\alpha+2\alpha\cos(\frac1{\sqrt{\alpha}})}{\sin(\frac1{\sqrt{\alpha}})}\sin(\frac{x}{\sqrt{\alpha}})-2\alpha \cos(\frac{x}{\sqrt{\alpha}})-x^2+2\alpha\right)^2\ dx\right)^{\frac12}$$

$=G_1(\alpha).$\\
This equation has one solution in the interval $[0.1,0.3]$, as we can see in Figure \ref{alpha1}. Solving this equation we obtain $\alpha=0.223688785954835$, with absolute error less than $10^{-16}$.
So the required solution is
$$u(x,t)=\left(\frac{1-2\alpha+2\alpha\cos(\frac1{\sqrt{\alpha}})}{\sin(\frac1{\sqrt{\alpha}})}
\sin(\frac{x}{\sqrt{\alpha}})-2\alpha\cos(\frac{x}{\sqrt{\alpha}})-x^2+2\alpha\right)(t+1)^{-1},$$
$$\alpha=0.223688785954835.$$
\begin{center}
\begin{figure}[!htb]
\begin{minipage}[t]{0.45\linewidth}
\centering
\includegraphics[width=0.8\textwidth]{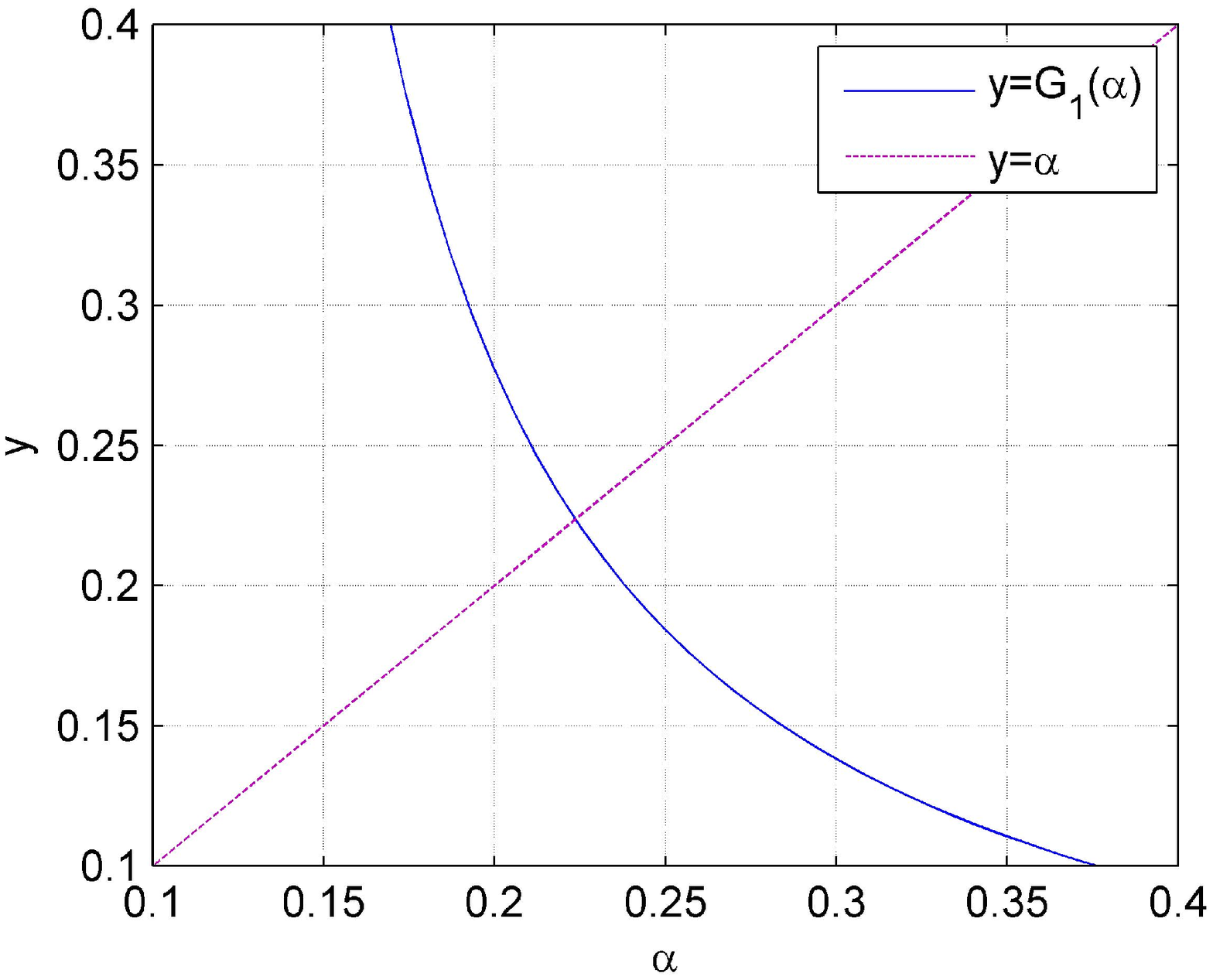}
\caption{Plot of the functions \mbox{$y=\alpha$} and $y=G_1(\alpha)$ for \mbox{example 1}.}\label{alpha1}
\end{minipage}\hfill
\begin{minipage}[t]{0.45\linewidth}
\centering
\includegraphics[width=0.8\textwidth]{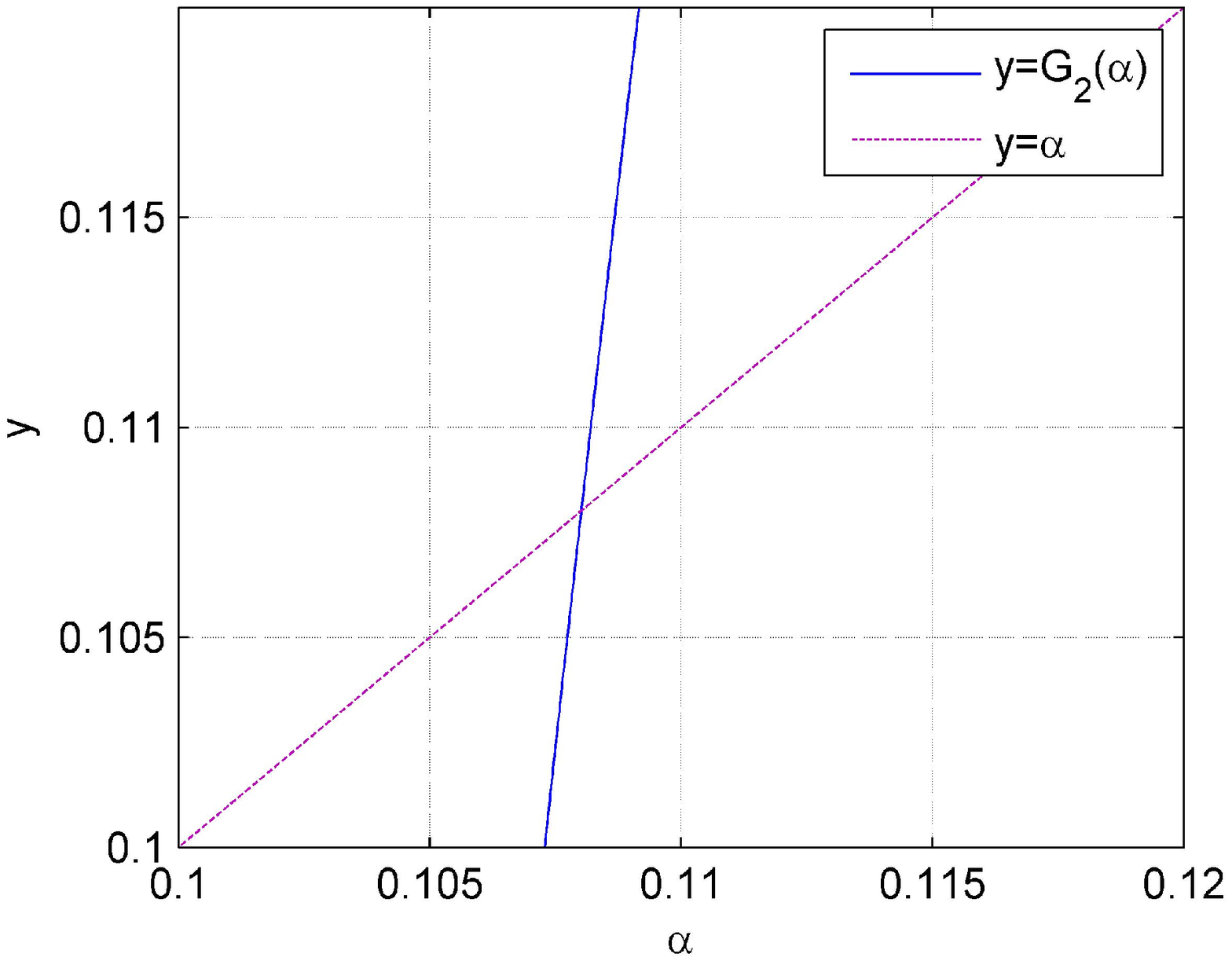}
\caption{Plot of the functions \mbox{$y=\alpha$} and $y=G_2(\alpha)$ for \mbox{example 2}.}\label{alpha2}
\end{minipage}
\end{figure}
\end{center}

\subsection{Example 2}
We now consider Problem (\ref{probi}) in $]0,1[$ with $\gamma=-\frac13$ and $f=e^{x}\sqrt{1-t}$,
\begin{equation}  \label{prob2}
\left\{
\begin{array}{l}
\displaystyle u_t- \left(\int_0^1u^2\ dx\right)^{-\frac13}u_{xx}=e^{x}\sqrt{1-t}\,, \quad (x,t)\in ]0,1[\times ]0,1],\\
\displaystyle u(0,t)=u(1,t)=0\,,\quad t\in ]0,1], \\
\displaystyle u(x,0)=u_{0}(x)\,, \quad x\in ]0,1[. \\
\end{array}
\right. \,
\end{equation}
For $\gamma=-\frac13$
$$l(t)=\left(-\frac23t+\frac23C\right)^{\frac32},$$
and so
$$f(x,t)=-g(x)\left(-\frac23t+\frac23C\right)^{\frac12}=e^{x}\sqrt{1-t}.$$
Thus $C=1$ and $g(x)=-\sqrt{\frac32}e^x$.
Solving the equation
$$w(x)+\alpha w''(x)=-\sqrt{\frac32}e^x,$$
we obtain the solution
$$w(x)=\left(C_1+\frac{\sqrt{\alpha}\sqrt{\frac32}}{\alpha+1}\right)\sin\left(\frac{x}{\sqrt{\alpha}}\right)+
\left(C_2+\frac{\sqrt{\frac32}}{\alpha+1}\right)\cos\left(\frac{x}{\sqrt{\alpha}}\right)-\frac{\sqrt{\frac32}}{\alpha+1}e^x$$
By the boundary conditions,
$$u(0,t)=0\Rightarrow C_2=0,$$
and\\
$$u(1,t)=0\Rightarrow C_1=\frac{e-\sqrt{\alpha}\sin\left(\frac{1}{\sqrt{\alpha}}\right)-\cos\left(\frac{1}{\sqrt{\alpha}}\right)}
{(\alpha+1)\sqrt{\frac23}\sin\left(\frac{1}{\sqrt{\alpha}}\right)}.$$
The equation for $\alpha $ is
$$\alpha=\left(\int_0^1\left(\frac{e-\sqrt{\frac32}\cos(\frac1{\sqrt{\alpha}})}{(\alpha+1)\sin(\frac1{\sqrt{\alpha}})}\sin(\frac{x}{\sqrt{\alpha}})+
\frac{\sqrt{\frac32}}{\alpha+1} \cos(\frac{x}{\sqrt{\alpha}})-\frac{\sqrt{\frac32}}{\alpha+1}e^x\right)^2\ dx\right)^{\frac12}$$

$=G_2(\alpha).$\\
As we can see in Figure \ref{alpha2}, this equation has one solution in the interval $[0.1,0.12]$.
Solving this equation, with the absolute error less than $10^{-16}$, we obtain
$$\alpha=0.108016681670528.$$
Hence the solution we were looking for is
$$u(x,t)=\left(\frac{e-\sqrt{\frac32}\cos(\frac1{\sqrt{\alpha}})}{(\alpha+1)\sin(\frac1{\sqrt{\alpha}})}\sin(\frac{x}{\sqrt{\alpha}})+
\frac{\sqrt{\frac32}}{\alpha+1} \cos(\frac{x}{\sqrt{\alpha}})-\frac{\sqrt{\frac32}}{\alpha+1}e^x\right)$$
$$\times\left(-\frac23t+\frac23\right)^{\frac32},$$
with $\alpha=0.108016681670528$.

\subsection{Example 3}
As a another example, we choose a 2D  problem, namely Problem (\ref{probi}) in $\Omega=]0,1[^2$ with $\gamma=2$ and $f=0$,
\begin{equation}  \label{prob3}
\left\{
\begin{array}{l}
\displaystyle u_t- \left(\int_{\Omega}u^2\ dx dy\right)^{2}\Delta u=0\,, \quad (x,t)\in \Omega\times ]0,1],\\
\displaystyle u(x,t)=0\,,\quad (x,t)\in \partial \Omega\times]0,1], \\
\displaystyle u(x,0)=u_{0}(x)\,, \quad x\in \Omega.\\
\end{array}
\right. \,
\end{equation}
For $\gamma=2$, $l(t)=(4t-4C)^{-\frac{1}{4}}$, and $g(x,y)=0$ because $f(x,y,t)=0$.
The factor $w(x,y)$ is defined by Equation (\ref{sol_w2d}) and, by
the boundary conditions,
$$u(0,y,t)=0\Rightarrow X(0)=0 \Rightarrow A_1=0,$$
$$u(x,0,t)=0\Rightarrow Y(0)=0 \Rightarrow B_1=0,$$
$$u(1,y,t)=0\Rightarrow X(1)=0 \Rightarrow A_2\sin\left(\sqrt{\frac{\lambda}{\alpha}}\right)=0 \Leftarrow \lambda=\pi^2\alpha,$$
$$u(x,1,t)=0\Rightarrow Y(1)=0 \Rightarrow B_2\sin\left(\sqrt{\frac{1-\pi^2\alpha}{\alpha}}\right)=0 \Leftarrow \alpha=\frac{1}{2\pi^2}.$$
Notice that $0<\lambda<1$, as assumed in Section 5.2. Therefore
$$w(x,y)=-A_2B_2\sin(\pi x)\cos(\pi y)=C_3\sin(\pi x)\cos(\pi y).$$
Using Equation (\ref{eq_alpha}) to calculate $C_3$, we obtain
$$\frac{1}{2\pi^2}=\left(\int_0^1\int_0^1 (C_3\sin(\pi x)\cos(\pi y))^2\ dxdy\right)^2\Leftrightarrow \frac{1}{2\pi^2}=\frac{C_3^4}{16}\Leftrightarrow C_3=\sqrt[4]{\frac{8}{\pi^2}}.$$
Finally,
$$u(x,y,t)=\frac{\sqrt[4]{\frac{8}{\pi^2}}\sin(\pi x)\sin(\pi y) }{\sqrt[4]{4t-4C}}.$$
\section{Numerical simulations}
\subsection{Example 1}
Consider the problem
\begin{equation*}
\left\{
\begin{array}{l}
\displaystyle u_t- \left(\int_0^1u^2\ dx\right)^{\frac12}u_{xx}=\frac{x^2}{(t+1)^2}\,, \quad (x,t)\in ]0,1[\times ]0,10],\\
\displaystyle u(0,t)=u(1,t)=0\,,\quad t\in ]0,10], \\
\displaystyle u(x,0)=u_0(x)\,, \quad x\in ]0,1[, \\
\end{array}
\right. \,
\end{equation*}
with $$u_0=\frac{1-2\alpha+2\alpha\cos(\frac1{\sqrt{\alpha}})}{\sin(\frac1{\sqrt{\alpha}})}
\sin(\frac{x}{\sqrt{\alpha}})-2\alpha\cos(\frac{x}{\sqrt{\alpha}})-x^2+2\alpha,$$
and $\alpha=0.223688785954835.$
In Figure \ref{ex1_sol}, we show the solution for $h=10^{-2}$, $\delta=10^{-3}$ and $k=2$. As expected, we can observe the decay of the solution as time increases.
\begin{center}
\begin{figure}[!htb]
\begin{minipage}[t]{0.30\linewidth}
\centering
\includegraphics[width=0.9\textwidth]{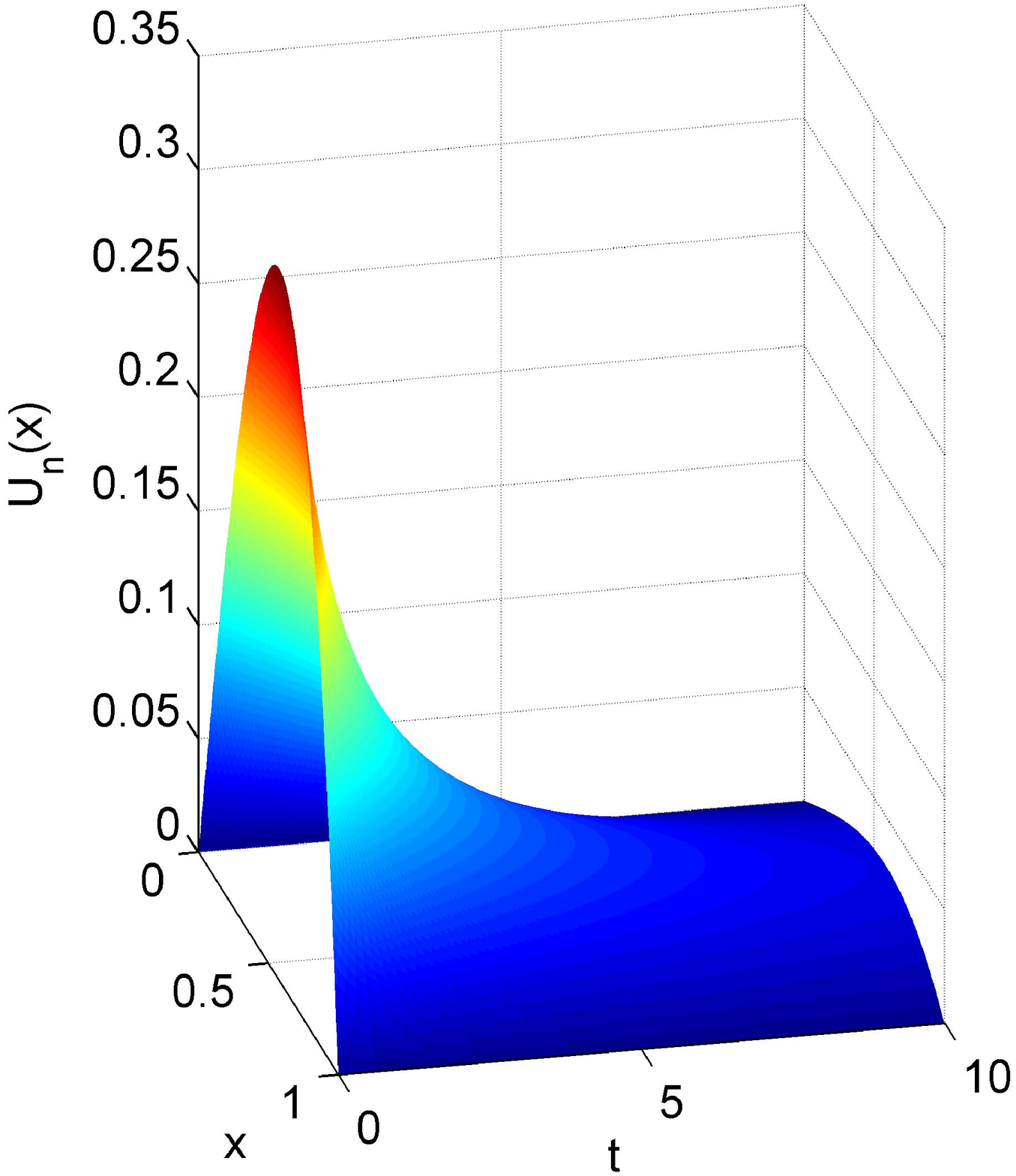}
\caption{Evolution of the obtained solution for $h=10^{-2}$, $\delta=10^{-3}$ and $k=2$.}\label{ex1_sol}
\end{minipage}\hfill
\begin{minipage}[t]{0.30\linewidth}
\centering
\includegraphics[width=0.9\textwidth]{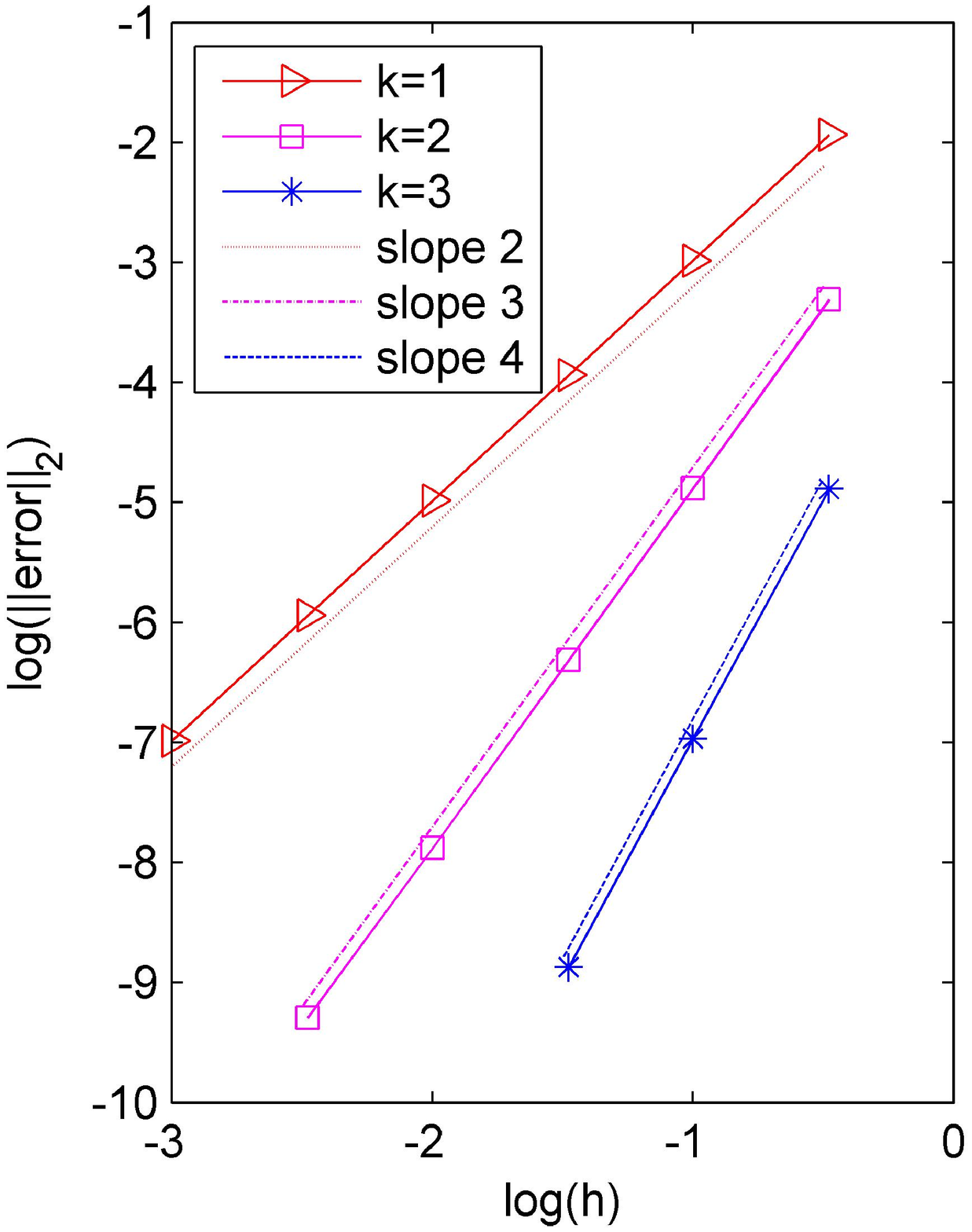}
\caption{Study of convergence for $h$, in example 1.}\label{ex1_O_h}
\end{minipage}\hfill
\begin{minipage}[t]{0.30\linewidth}
\centering
\includegraphics[width=0.9\textwidth]{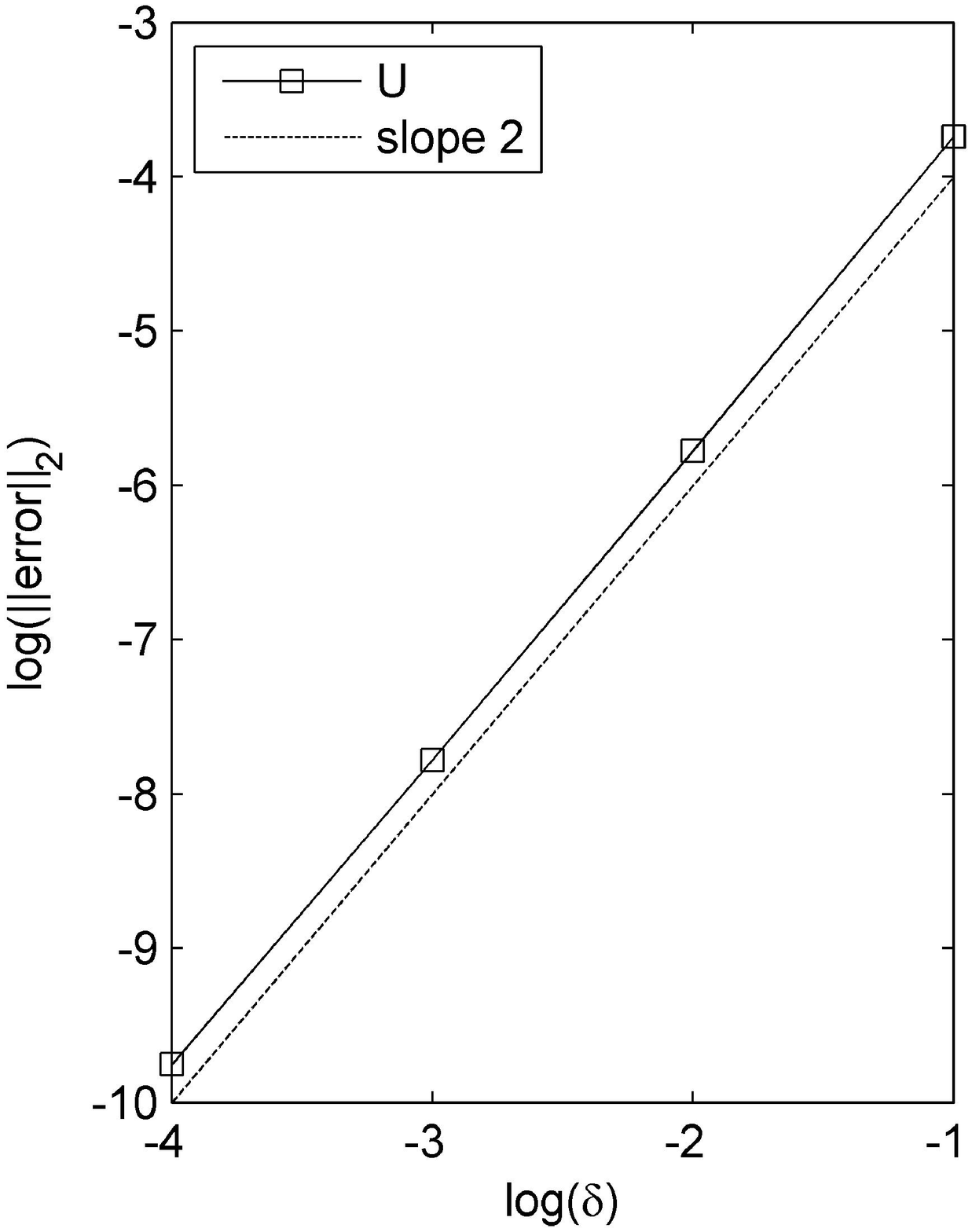}
\caption{Study of convergence for $\delta$, in example 1.}\label{ex1_O_d}
\end{minipage}
\end{figure}
\end{center}
In order to study the order of convergence for $h$, we made several runs with different combinations of $h$ and $k$. Since we know the exact solution, we calculated the $L_2$ norm of the errors at $t=10$, and we plotted the logarithm of the errors versus the logarithm of $h$ in Figure \ref{ex1_O_h}. It is evident that the second order of convergence for $k$ equals one, the third order for $k$ equal two and a fourth order for $k$ equals three. We repeated the procedure for $\delta$, which is illustrated in Figure  \ref{ex1_O_d}, and we concluded the second order of convergence, as expected.

\subsection{Example 2}
Consider Problem (\ref{prob2})
\begin{equation*}
\left\{
\begin{array}{l}
\displaystyle u_t- \left(\int_0^1u^2\ dx\right)^{-\frac13}u_{xx}=e^{x}\sqrt{[1-t]_+}\,, \quad (x,t)\in ]0,1[\times ]0,2],\\
\displaystyle u(0,t)=u(1,t)=0\,,\quad t\in ]0,1], \\
\displaystyle u(x,0)=u_0\,, \quad x\in ]0,1[, \\
\end{array}
\right. \,
\end{equation*}
with
$$u_0(x)=\left(\frac{e-\sqrt{\frac32}\cos(\frac1{\sqrt{\alpha}})}{(\alpha+1)\sin(\frac1{\sqrt{\alpha}})}\sin(\frac{x}{\sqrt{\alpha}})+
\frac{\sqrt{\frac32}}{\alpha+1} \cos(\frac{x}{\sqrt{\alpha}})-\frac{\sqrt{\frac32}}{\alpha+1}e^x\right)\left(\frac23\right)^{\frac32},$$
and $\alpha=0.108016681670528$.
In Figure \ref{ex2_sol}, we show the solution obtained for $h=10^{-2}$, $\delta=10^{-3}$ and $k=2$. As expected we can observe an extinction in $t=1$. This effect is more evident in the graph of Figure \ref{comp_ass}, where we plotted the energetic function $\log(\int_{\Omega} U(x,t)^2\ dx)$ for the three examples.
\begin{center}
\begin{figure}[!htb]
\centering
\includegraphics[width=0.3\textwidth]{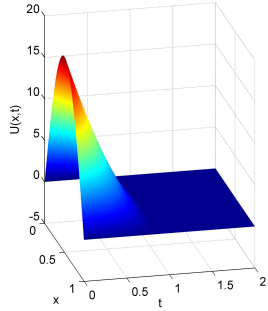}
\caption{Evolution in time of the solution obtained for $h=10^{-2}$, $\delta=10^{-3}$ and $k=2$.}\label{ex2_sol}
\end{figure}
\end{center}

\subsection{Example 3}
Consider Problem (\ref{prob3}) with $C=-\frac14$,
\begin{equation*}
\left\{
\begin{array}{l}
\displaystyle u_t- \left(\int_{\Omega}u^2\ dx dy\right)^{2}\Delta u=0\,, \quad (x,t)\in \Omega\times ]0,1],\\
\displaystyle u(x,t)=0\,,\quad (x,t)\in \partial \Omega\times]0,1], \\
\displaystyle u(x,0)=\sqrt[4]{\frac{8}{\pi^2}}\sin(\pi x)\sin(\pi y)\,, \quad x\in \Omega.\\
\end{array}
\right. \,
\end{equation*}
This problem was simulated with polynomial approximations of degree 3 in $x$ and $y$. In Figure \ref{ex3_sol}, we plotted the solution obtained, when $h=0.0625$ and $\delta=10^{-2}$, for some values of $t$. As expected, the solution decays with time.\\
\begin{center}
\begin{figure}
\centering
\includegraphics[width=0.3\textwidth]{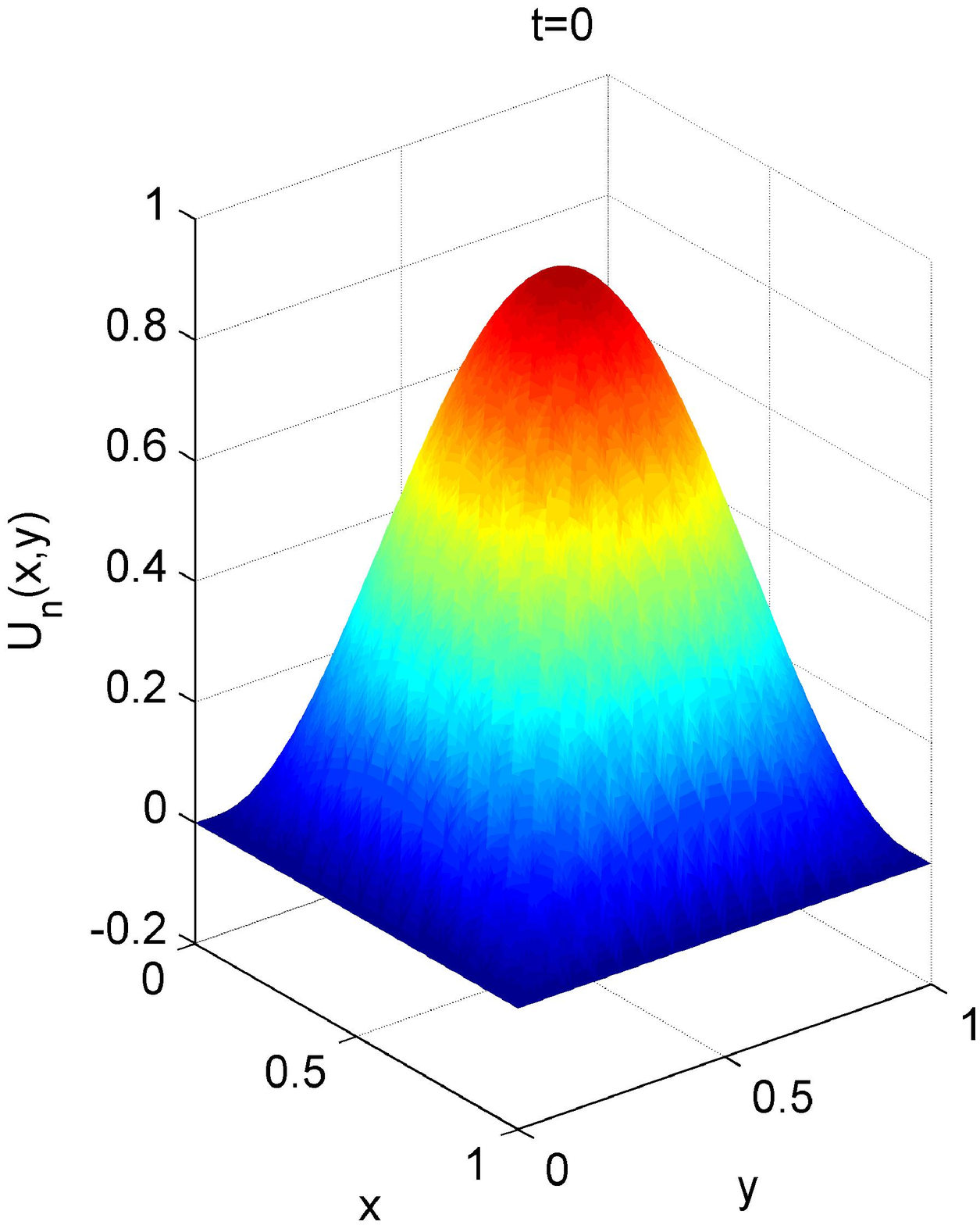}
\includegraphics[width=0.3\textwidth]{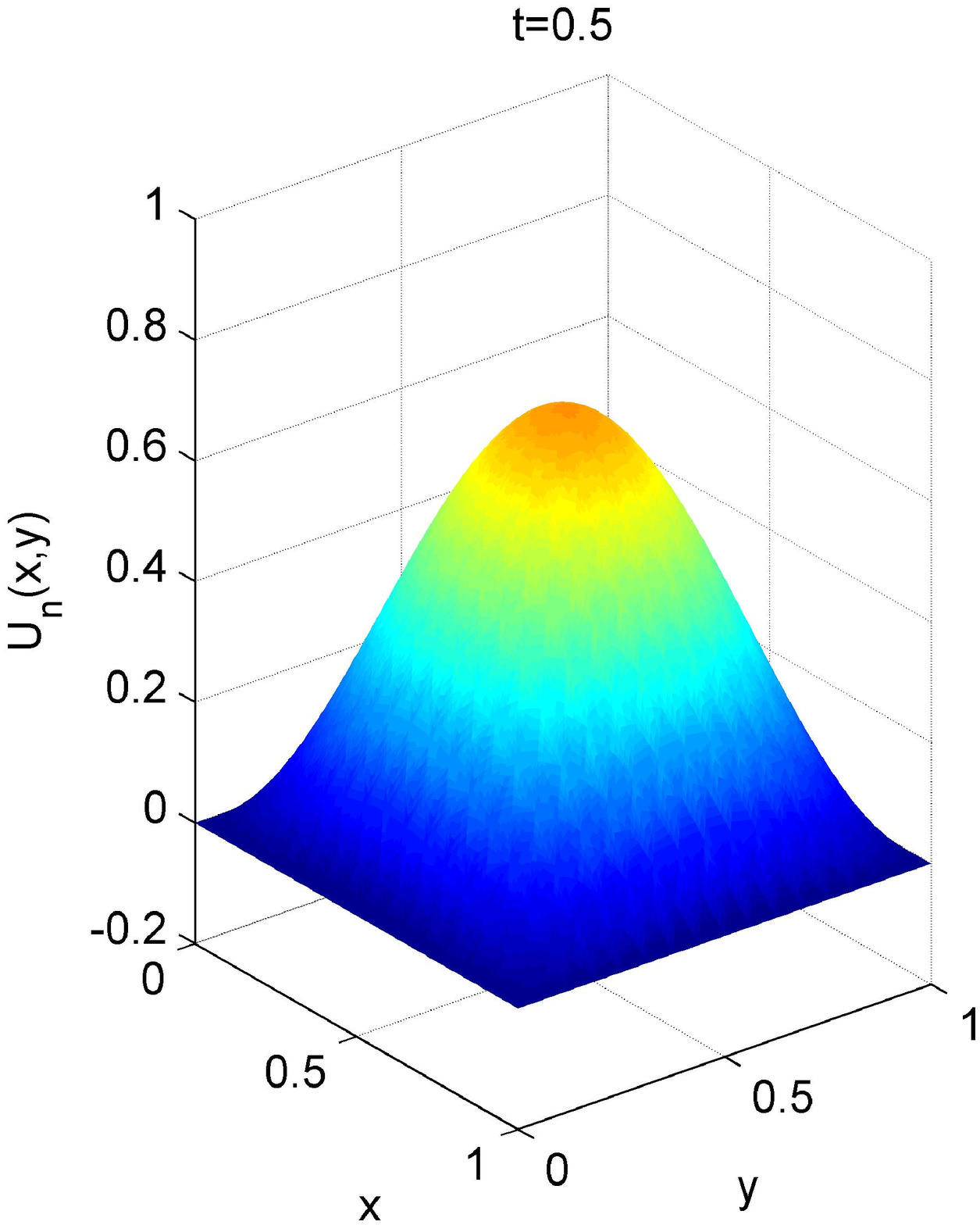}
\includegraphics[width=0.3\textwidth]{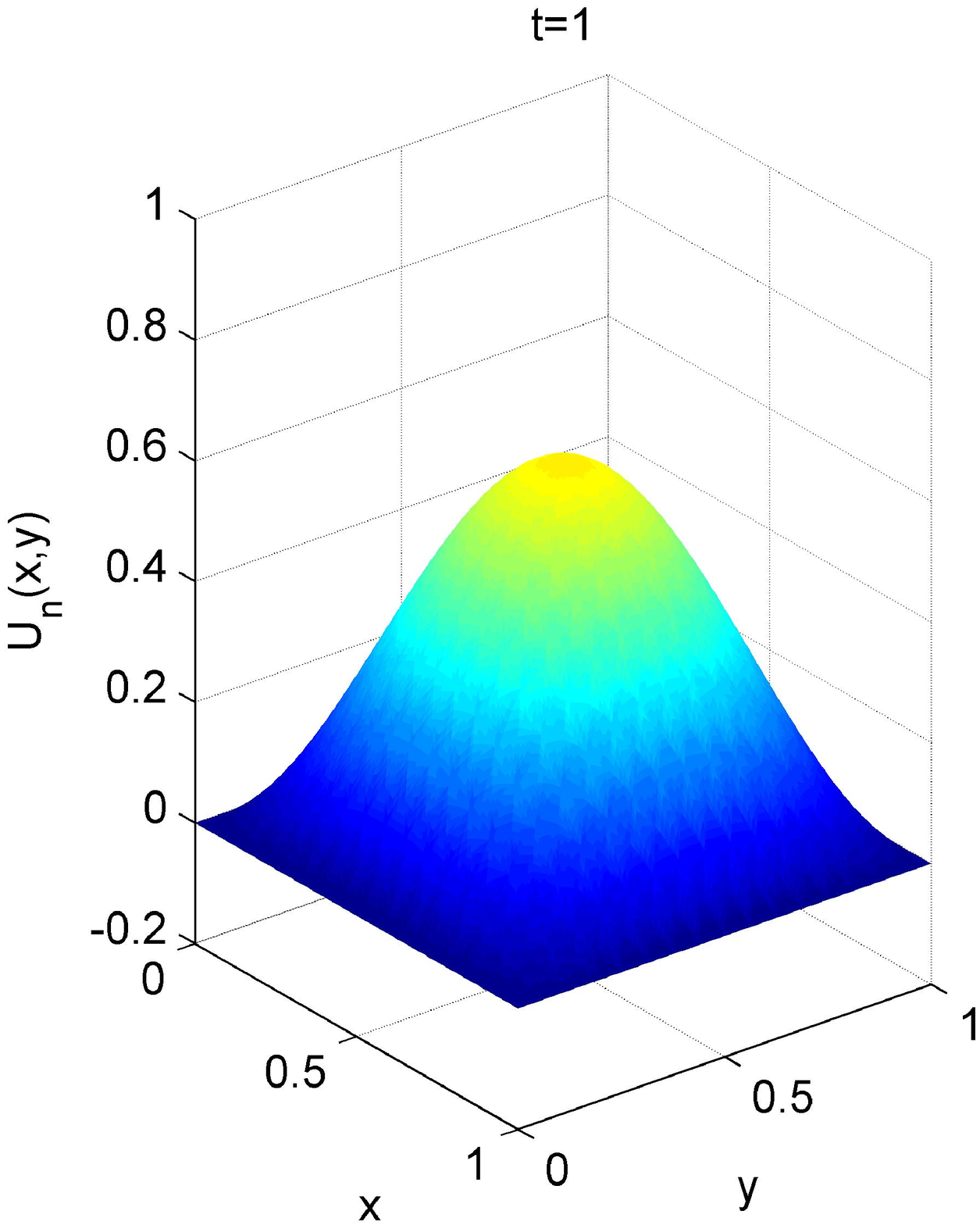}
\caption{The obtained solution in some values of $t$.}\label{ex3_sol}
\end{figure}
\end{center}
In the same way as in Example 1, we made a study of the numerical convergence in $h$ and $\delta$. The results of the several runs made with different combinations of  $k$, $h$ and $\delta$ are plotted in Figures \ref{ex3_O_h} and \ref{ex3_O_d}. The existence of a known exact solution permitted us to calculate the norm in $L_2(\Omega)$ of the errors at $t=1$. The analysis of the figures agrees with the convergence orders proved in Theorem \ref{conv_d}.
\begin{center}
\begin{figure}[!htb]
\begin{minipage}[t]{0.3\linewidth}
\centering
\includegraphics[width=0.9\textwidth]{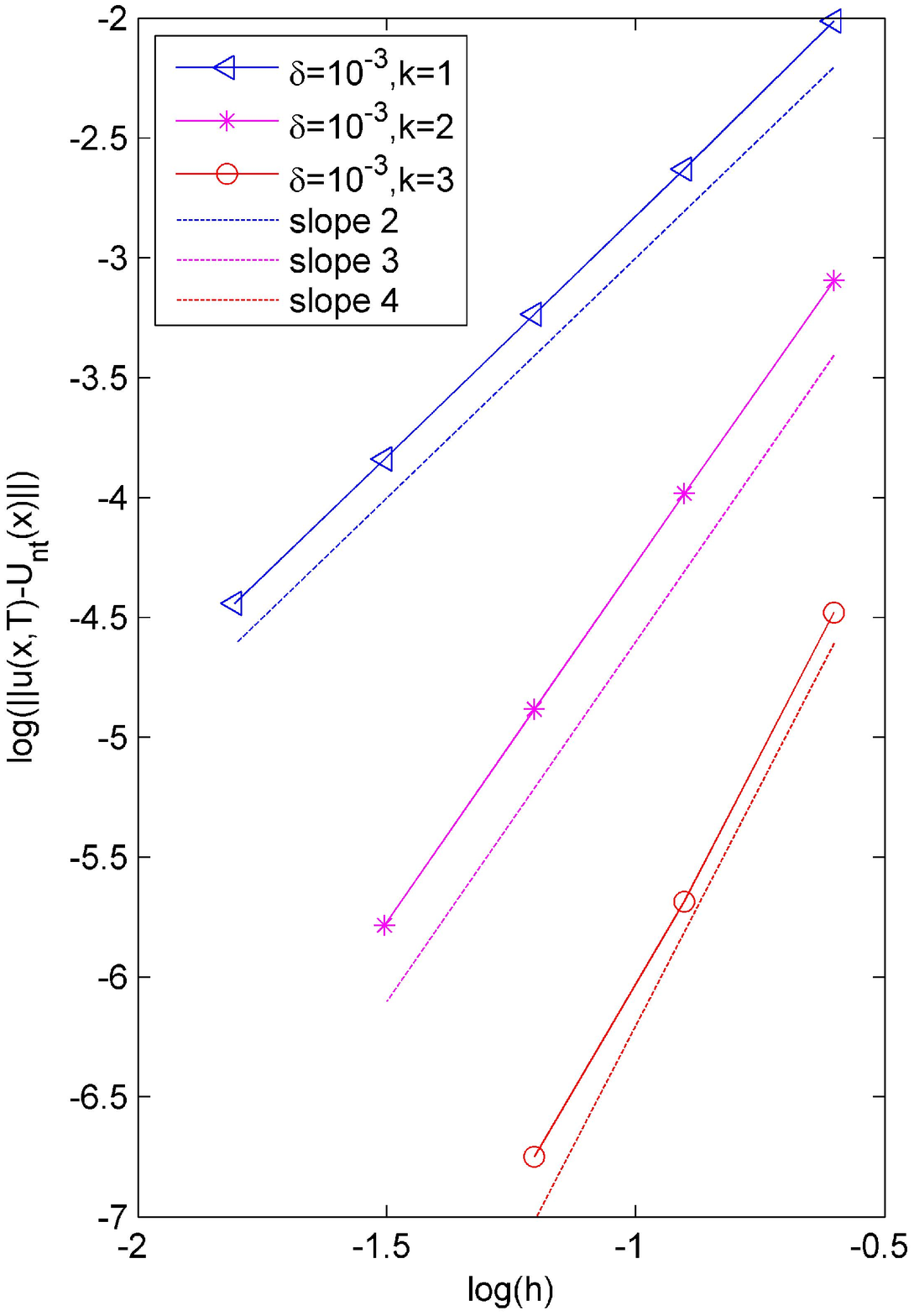}
\caption{Study of convergence for $h$, in Example 3.}\label{ex3_O_h}
\end{minipage}\hfill
\begin{minipage}[t]{0.3\linewidth}
\centering
\includegraphics[width=0.9\textwidth]{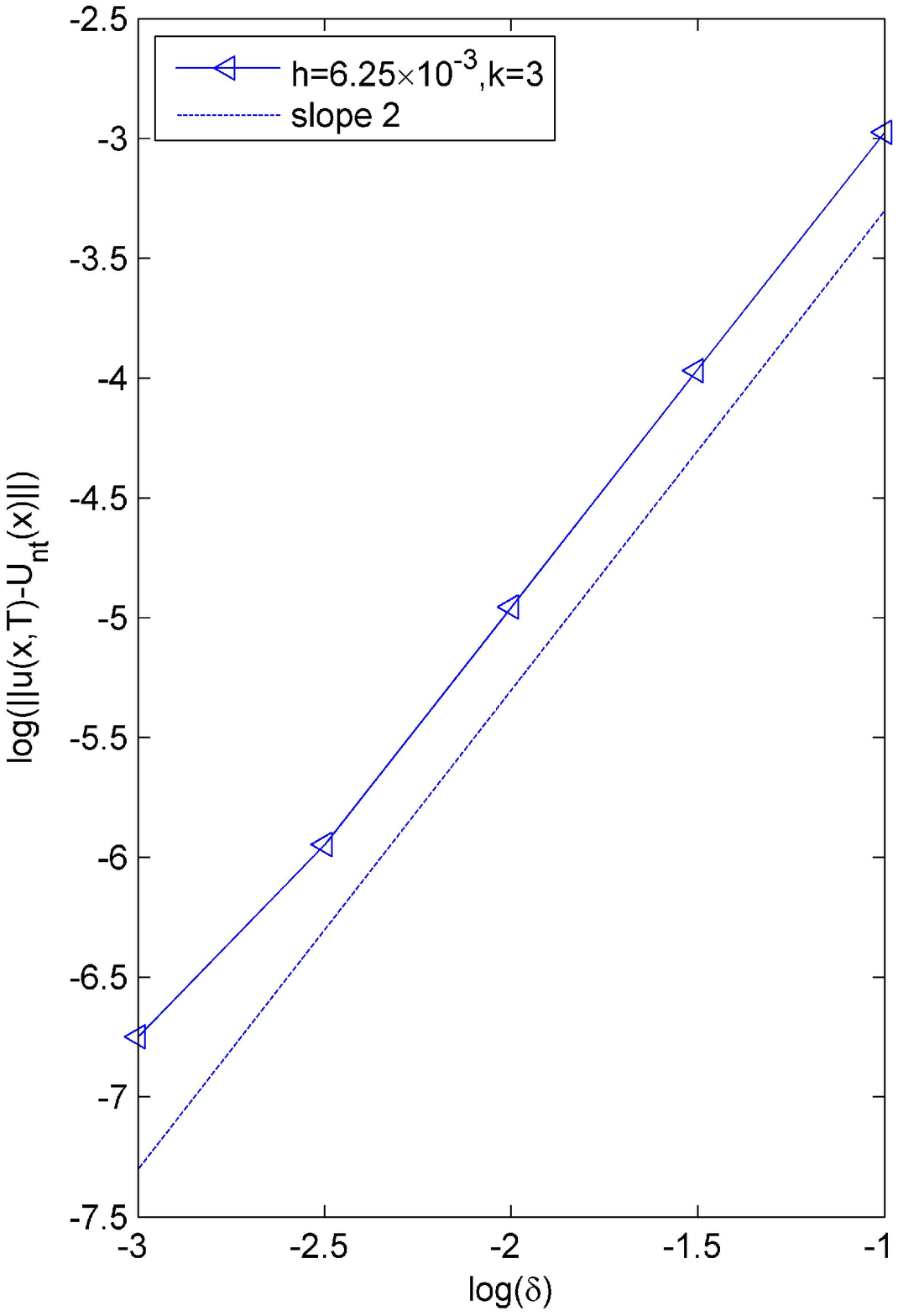}
\caption{Study of convergence for $\delta$, in Example 3.}\label{ex3_O_d}
\end{minipage}
\hfill
\begin{minipage}[t]{0.3\linewidth}
\centering
\includegraphics[width=0.9\textwidth]{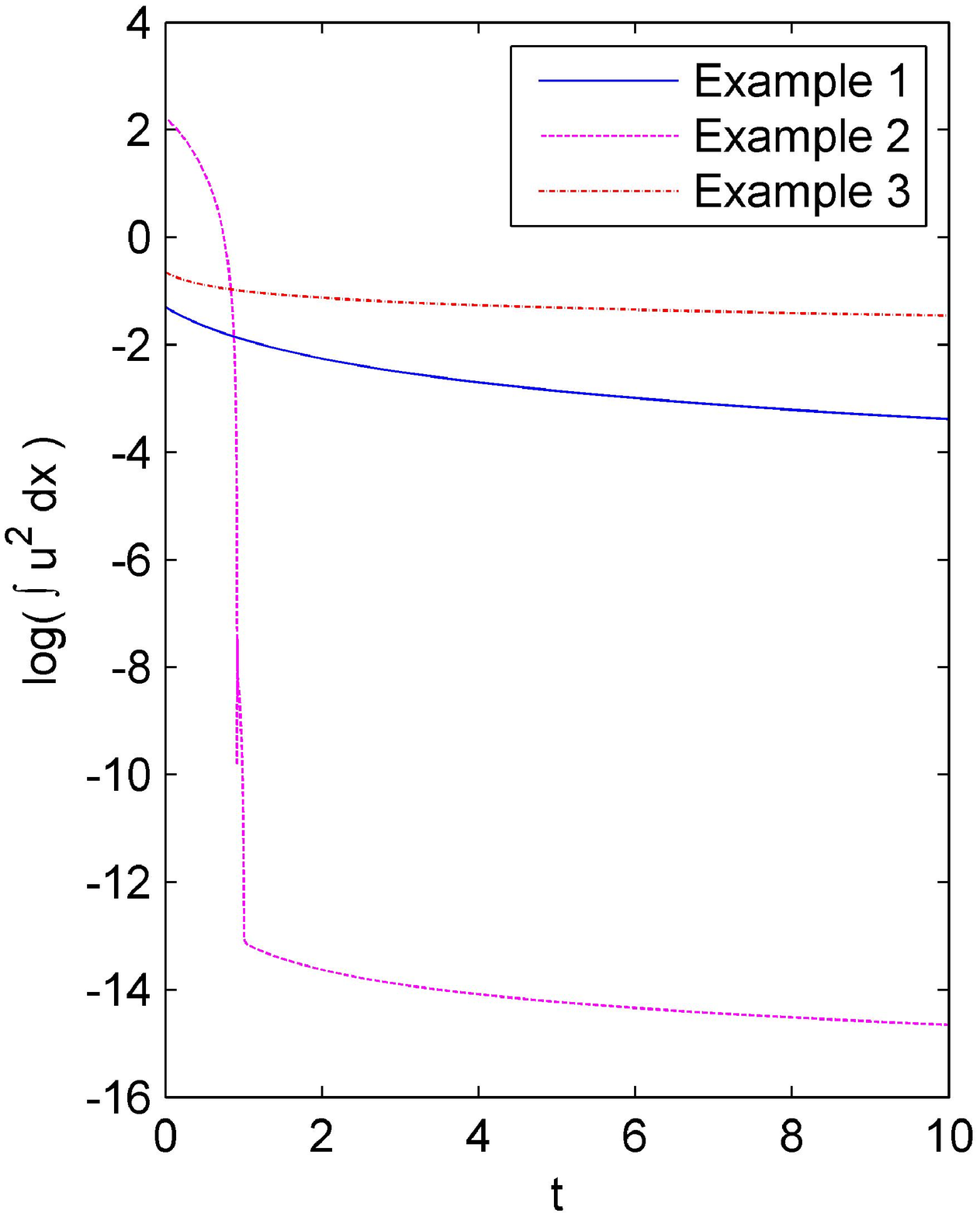}
\caption{Study of the asymptotic behaviour.}\label{comp_ass}
\end{minipage}
\end{figure}
\end{center}

\section{Conclusions}

We proved optimal rates of convergence for a linearised
Crank-Nicolson-Galerkin finite element method with piecewise polynomials of
arbitrary degree basis functions in space when applied to a degenerate
nonlocal parabolic equation. Some numerical experiments were presented,
considering different functions $f$ and exponent $\gamma$. The
numerical results are in agreement with the exact explicit solutions deduced,  and in accordance with the theoretical results.

\section*{Acknowledgements}

This work was partially supported by the research projects:\\ OE/MAT/UI0212/2014 - financed by FEDER through the - Programa Operacional Factores de Competitividade, FCT - Funda\c{c}\~ao para a Ci\^{e}ncia e a Tecnologia, Portugal and MTM2011-26119, MICINN, Spain.


\end{document}